\documentclass[12pt,reqno]{amsart}
\usepackage{amssymb,amsmath,amsthm}
\usepackage{mathtools}
\usepackage{graphicx}
\usepackage{subcaption}
\usepackage{fullpage}
\usepackage{scrextend}
\usepackage{siunitx}
\usepackage{enumerate}
\usepackage{tikz,calc}
\usepackage{tikz-cd}
\usetikzlibrary{automata,positioning}
\usepackage{comment}
\usetikzlibrary{calc}
\usepackage{epsfig,graphicx}
\usepackage{listings}
\usepackage{enumitem}

\usepackage{float}
\usepackage{bbm}
\usepackage[colorlinks=true, pdfstartview=FitH, linkcolor=blue, citecolor=blue, urlcolor=blue]{hyperref}

\newcommand{\R}{\mathbb{R}}

\newcommand{\Mod}[1]{\ (\text{mod}\ #1)}
\renewcommand{\Mod}[1]{{\ifmmode\text{\rm\ (mod~$#1$)}\else\discretionary{}{}{\hbox{ }}\rm(mod~$#1$)\fi}}

\newcommand{\g}{\mathcal{G}_{k_1,k_2,k_3}}
\newcommand{\skel}{B_{k_1,k_2,k_3}}

\newtheorem{theorem}{Theorem}[section]

\newtheorem{lemma}[theorem]{Lemma}
\newtheorem{cor}[theorem]{Corollary}

\theoremstyle{definition} % changes the formatting of the environments defined below
\newtheorem{defn}[theorem]{Definition}
\newtheorem{remark}[theorem]{Remark}
\numberwithin{theorem}{section}

\begin{document}

\title{An exponential bound for simultaneous embeddings of planar graphs}
\author{Ritesh Goenka}
\address{Department of Mathematics \\ University of British Columbia \\ Room 121, 1984 Mathematics Road \\ Vancouver, BC, Canada V6T 1Z2}
\email{rgoenka@math.ubc.ca}
\author{Pardis Semnani}
\address{Department of Mathematics \\ University of British Columbia \\ Room 121, 1984 Mathematics Road \\ Vancouver, BC, Canada V6T 1Z2}
\email{psemnani@math.ubc.ca}
\author{Chi Hoi Yip}
\address{Department of Mathematics \\ University of British Columbia \\ Room 121, 1984 Mathematics Road \\ Vancouver, BC, Canada V6T 1Z2}
\email{kyleyip@math.ubc.ca}
\subjclass[2020]{05C62, 68R10}
\keywords{Graph drawing, planar graphs, simultaneous straight-line embeddings}

\maketitle

\begin{abstract}
We show that there are $O(n \cdot 4^{n/11})$ planar graphs on $n$ vertices which do not admit a simultaneous straight-line embedding on any $n$-point set in the plane. In particular, this improves the best known bound $O(n!)$ significantly.
\end{abstract}

\section{Introduction}

A planar graph $G$ admits a \emph{plane straight-line embedding} (also known as \emph{F\'ary embedding}) on a point set $X \subset \R^2$ if $G$ can be drawn by pairwise noncrossing line segments such that each vertex is represented by a point in $X$. F\'ary-Wagner theorem \cite{F48, W36} states that every planar graph admits a straight-line embedding on $\R^2$. There are many open problems related to straight-line embeddings; we refer to \cite[Section 9.2]{BMP05} for more discussions.

Of particular interest is the universal set problem. A point set $X$ in the Euclidean plane is said to be \emph{$n$-universal} if every planar graph with $n$ vertices admits a plane straight-line embedding on $X$. De Fraysseix, Pach, and Pollack \cite{DPP90} proved that the $(n-1) \times (n-1)$ integer grid is an $n$-universal set, and Chrobak and Payne \cite{CP95} provided a linear time algorithm to construct such an embedding. Let $f(n)$ be the minimum size of an $n$-universal set. The best-known upper bound on $f(n)$ is $n^2/4+\Theta(n)$, due to Bannister, Cheng, Devanny, and  Eppstein \cite{BCDE14}. It is clear that $f(n) \geq n$, and a brute-force computer search shows that $f(n)=n$ for $n \leq 10$ \cite{CHK15}. For sufficiently large $n$, this trivial lower bound has been improved subsequently in \cite{DPP90, CK89, K04} with the best known lower bound $f(n) \geq (1.293-o(1))n$ due to Scheucher, Schrezenmaier, and Steiner \cite{SSS20}.

Another problem, which is closely related, is the problem of simultaneous embedding of planar graphs, studied systematically in \cite{BCC07} (see also the survey by Bl\"asius~et al.~\cite{BKR13}). We follow the notation in \cite{BCC07, CHK15, SSS20}. Given a collection $\mathcal{G}_n$ of planar graphs on $n$ vertices, we say the collection is \emph{simultaneously embeddable} (without mapping) if there exists a point set $X \subset \R^2$ with size $n$, such that all graphs in  $\mathcal{G}_n$  admit a straight-line embedding on $X$; otherwise, we say $\mathcal{G}_n$ is a \emph{conflict collection}. Let $\sigma(n)$ be the minimum size of a conflict collection $\mathcal{G}_n$ of planar graphs on $n$ vertices.

Brass~et al. \cite[Section 7]{BCC07} asked if there exists $n$ such that $\sigma(n)= 2$. This question is widely open. Cardinal, Hoffmann, and Kusters \cite{CHK15} proved that $\sigma(35) \leq 7393$ by constructing a nice conflict collection. Using an exhaustive computer search, Scheuche~et al.~\cite{SSS20} found a conflict collection of $49$ graphs on $11$ vertices. On the other hand, to the best of our knowledge, the best known upper bound on $\sigma(n)$ for large $n$ is $\frac{1}{8}(5n+12)(n-1)!$ \cite[Lemma 4]{CHK15}, which is super-exponential in $n$.

We establish the following theorem, which gives an exponential upper bound on $\sigma(n)$.

\begin{theorem}\label{thm}
    For $n \ge 238$, there is a collection $\mathcal{G}_n$ of non-isomorphic planar graphs on $n$ vertices containing at least $\frac{1}{6} \cdot 7^{(2n-8)/11}$ graphs, no more than $(21n + 552) 4^{(n+37)/11}$ of which are simultaneously embeddable. In particular, $\sigma(n) = O(n \cdot 4^{n/11})$.
\end{theorem}

As a quick application, Theorem~\ref{thm} leads to a non-trivial lower bound on the size of an $n$-universal set, which is, unfortunately, worse than the best-known lower bound \cite{SSS20}.

\begin{cor}\label{cor}
    $f(n) \geq (1.059-o(1))n$ as $n \to \infty$.
\end{cor}

We describe our construction for graphs in the family $\mathcal{G}_n$ in Section~\ref{sec:construction}.
The proof of Theorem~\ref{thm} and Corollary~\ref{cor} is discussed in Section~\ref{sec:proof}.

\section{Construction}
\label{sec:construction}

In this section, we construct a family of graphs, any sufficiently large subset of which we will later show cannot be simultaneously embedded.
\begin{defn}[{\cite[Section~2]{CHK15}}]
    A \emph{planar $3$-tree} (also known as a \emph{stacked triangulation}) is a maximal planar graph obtained by iteratively splitting a facial triangle into three new triangles with a degree-three vertex, starting from a single triangle.
\end{defn}
For each natural number $n \ge 22$, we construct a family $\mathcal{G}_n$ of planar graphs on $n$ vertices using planar $3$-trees.

Let $\mathcal{T} \coloneqq \{T_1,T_2,T_3,T_4,T_5,T_6,T_7\}$, where $T_1,\ldots,T_7$ are the graphs shown in Figure~\ref{fig:family of 7}.
Each of these graphs is a planar $3$-tree in which the face bounded by $a,b,c$ is designated as the outer face. For each $1 \leq i \leq 7$, the graph $T_i$ has $8$ vertices and $12$ faces. By \cite[Lemma 8]{CHK15}, no three of the graphs in $\mathcal{T}$ admit a simultaneous embedding with a fixed mapping for the outer face. More precisely, for any point set $X \subset \R^2$ consisting of $8$ points with convex hull equal to a triangle with vertices $A, B, C$, at most two graphs from $\mathcal{T}$ admit a plane straight-line embedding on $X$, where the vertices $a$, $b$ and $c$ are mapped to the points $A$, $B$ and $C$, respectively.

\begin{figure}[H]
    \centering
    \begin{subfigure}{0.23\textwidth}
        \centering
        \includegraphics[width=\textwidth]{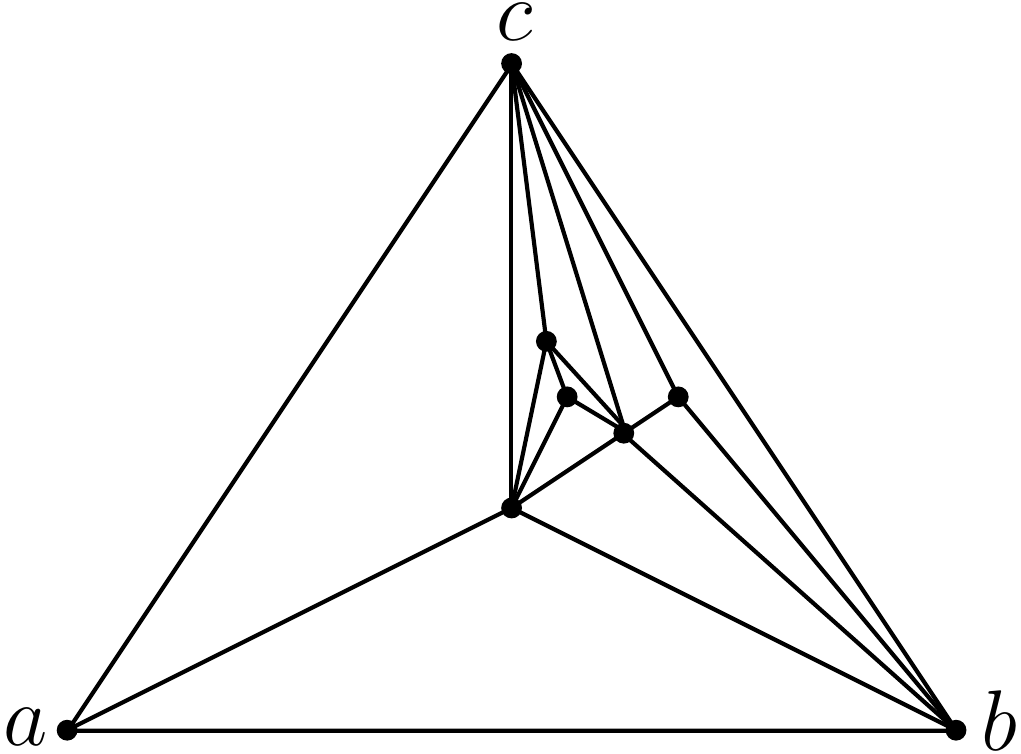}
        \caption{$T_1$}
    \end{subfigure}%
    \hfill
    \begin{subfigure}{0.23\textwidth}
        \centering
        \includegraphics[width=\textwidth]{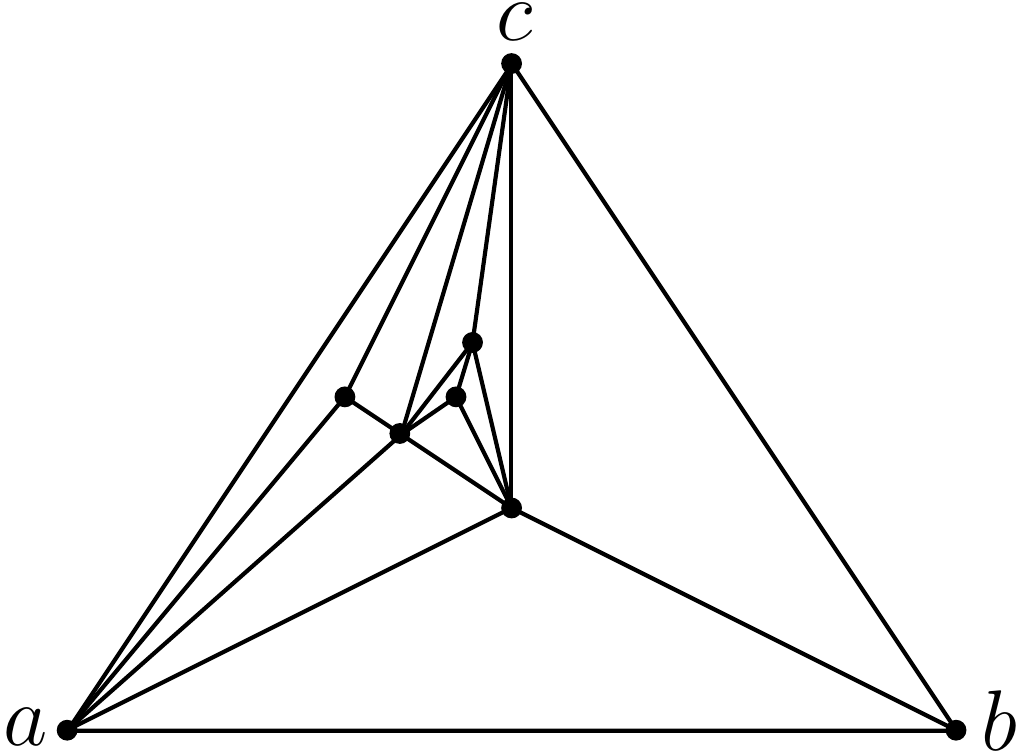}
        \caption{$T_2$}
    \end{subfigure}%
    \hfill
    \begin{subfigure}{0.23\textwidth}
        \centering
        \includegraphics[width=\textwidth]{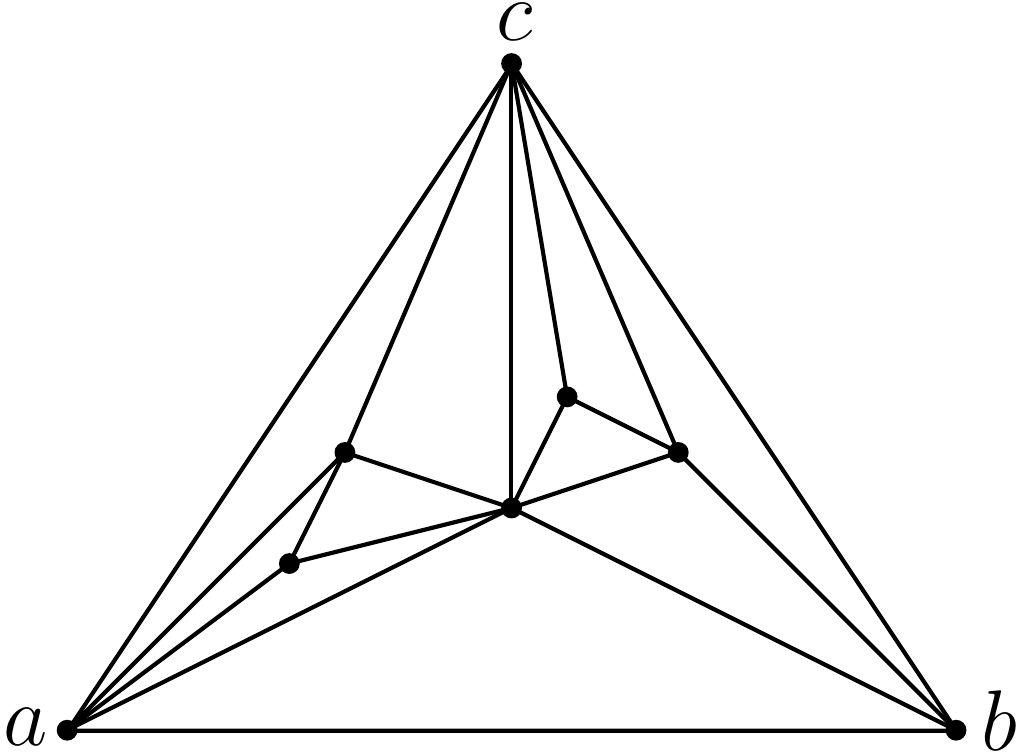}
        \caption{$T_3$}
    \end{subfigure}%
    \hfill
    \begin{subfigure}{0.23\textwidth}
        \centering
        \includegraphics[width=\textwidth]{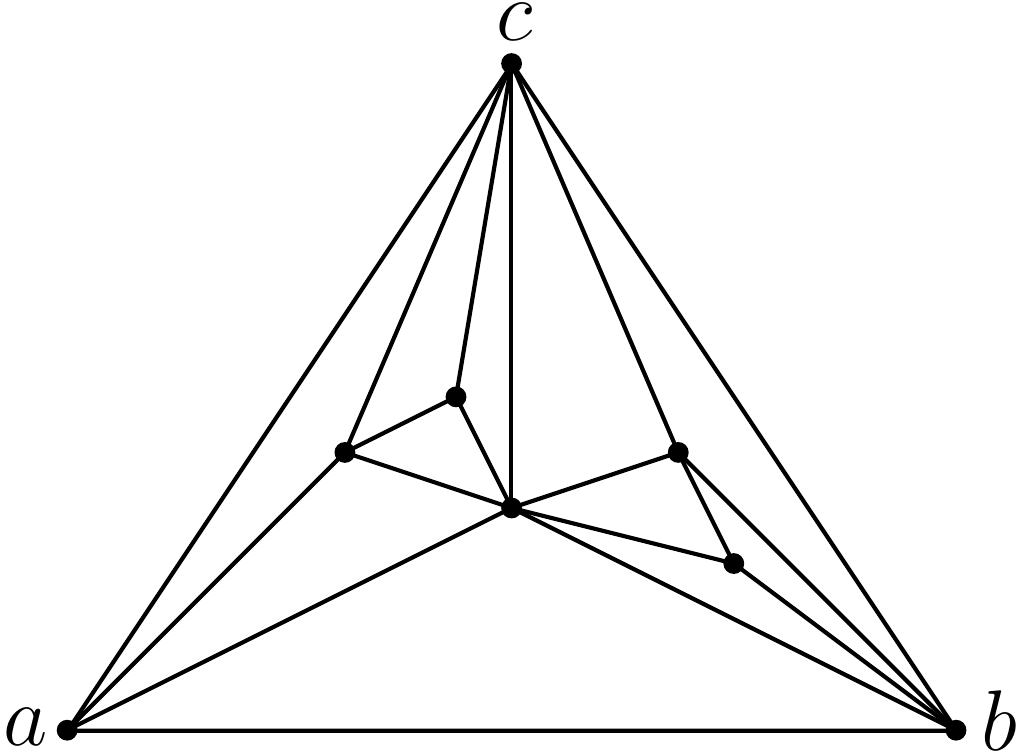}
        \caption{$T_4$}
    \end{subfigure}%
    \hfill
    \vspace{2mm}
    \begin{subfigure}{0.23\textwidth}
        \centering
        \includegraphics[width=\textwidth]{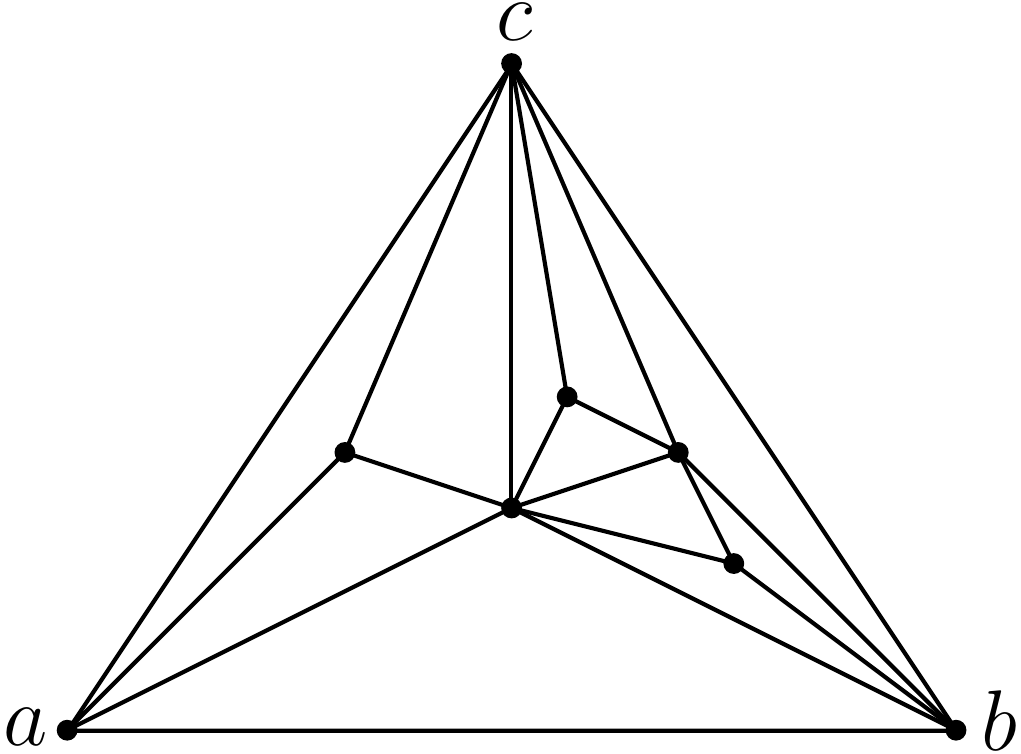}
        \caption{$T_5$}
    \end{subfigure}%
    \hfill
    \begin{subfigure}{0.23\textwidth}
        \centering
        \includegraphics[width=\textwidth]{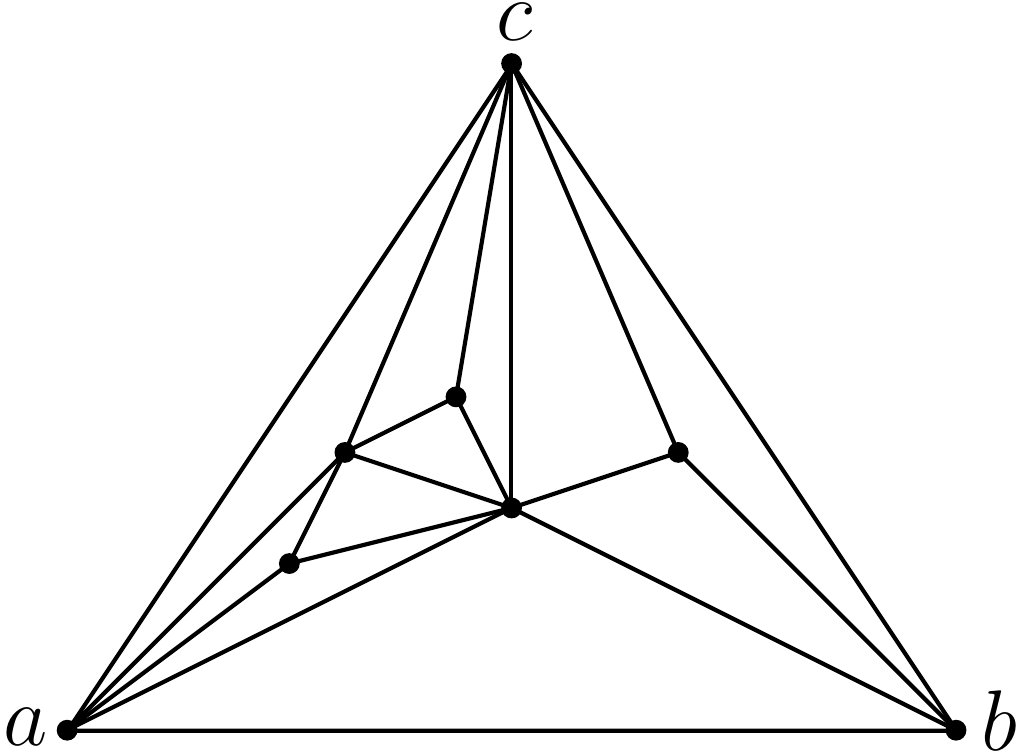}
        \caption{$T_6$}
    \end{subfigure}%
    \hfill
    \begin{subfigure}{0.23\textwidth}
        \centering
        \includegraphics[width=\textwidth]{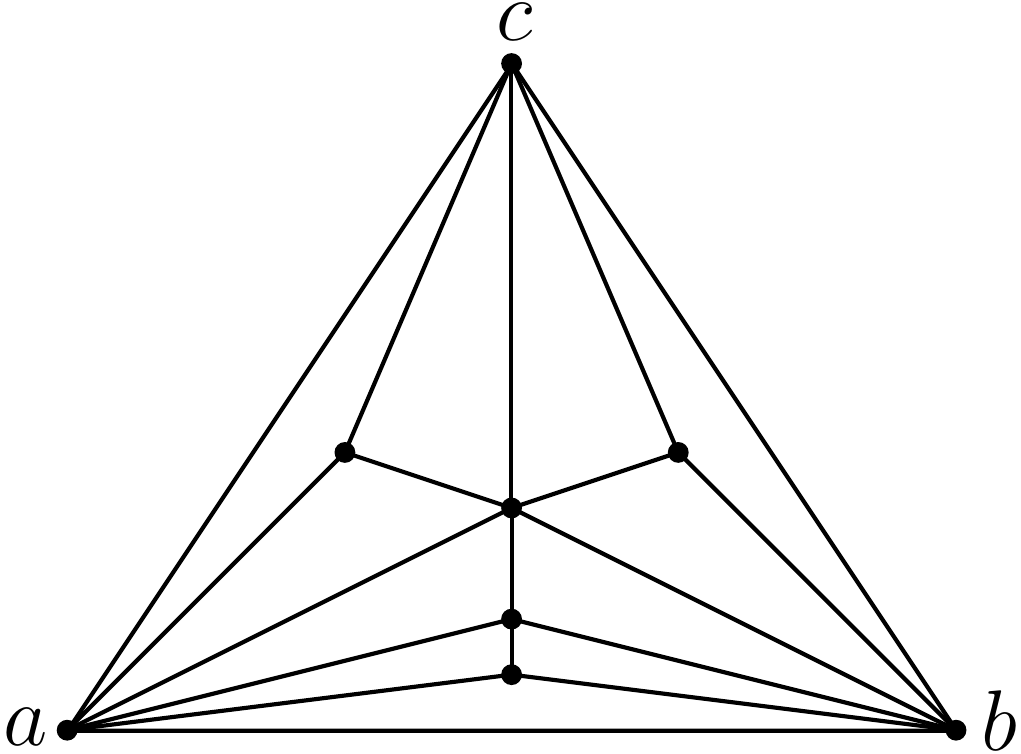}
        \caption{$T_7$}
    \end{subfigure}%
    \hfill
    \caption{Seven planar graphs, no three of which admit a simultaneous embedding with a fixed mapping for the outer face. This family is taken from \cite{CHK15}.}
    \label{fig:family of 7}
\end{figure}

Similarly, let $\widetilde{\mathcal{T}}\coloneqq \{\widetilde{T}_1,\widetilde{T}_2,\widetilde{T}_3\}$, where $\widetilde{T}_1, \widetilde{T}_2, \widetilde{T}_3$ are the graphs depicted in Figure~\ref{fig2:family of 3}. Each of these graphs is again a planar $3$-tree with the face bounded by $a,b,c$ designated as its outer face. For each $1 \leq i \leq 3$, the graph $\widetilde{T}_i$ has 5 vertices and 6 faces. Cardinal et al. \cite[Lemma 7]{CHK15} proved that these graphs do not admit a simultaneous embedding with a fixed mapping for the outer face. Note that both $\mathcal{T}$ and $\widetilde{\mathcal{T}}$ are flip-symmetric, i.e., the graph obtained by interchanging vertices $a$ and $b$ in any graph in the family is also in the family.

\begin{figure}[H]
    \centering
    \begin{subfigure}{.3\textwidth}
        \centering
        \includegraphics[width=\textwidth]{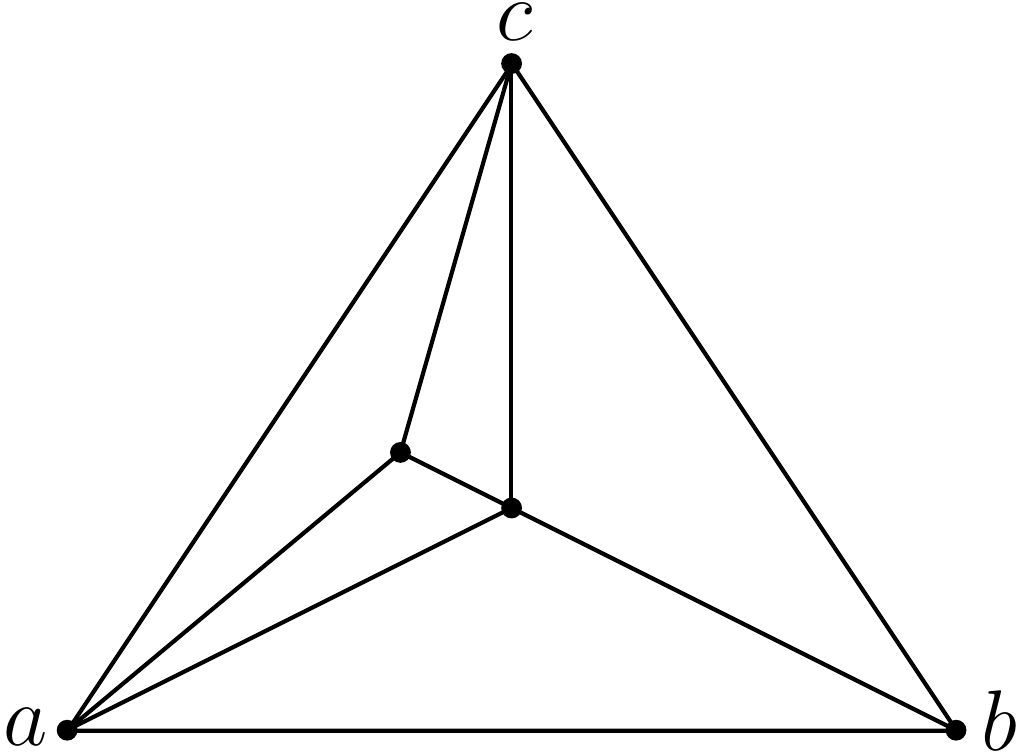}
        \caption{$\widetilde{T}_1$}
    \end{subfigure}%
    \hfill
    \begin{subfigure}{.3\textwidth}
        \centering
        \includegraphics[width=\textwidth]{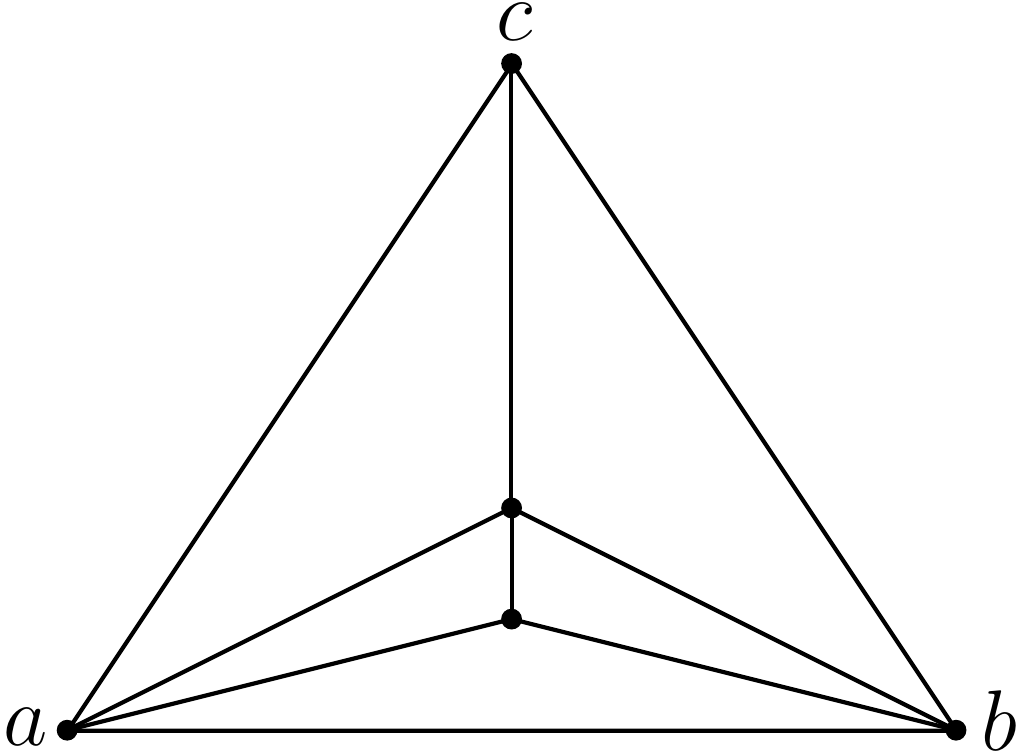}
        \caption{$\widetilde{T}_2$}
    \end{subfigure}%
    \hfill
    \begin{subfigure}{.3\textwidth}
        \centering
        \includegraphics[width=\textwidth]{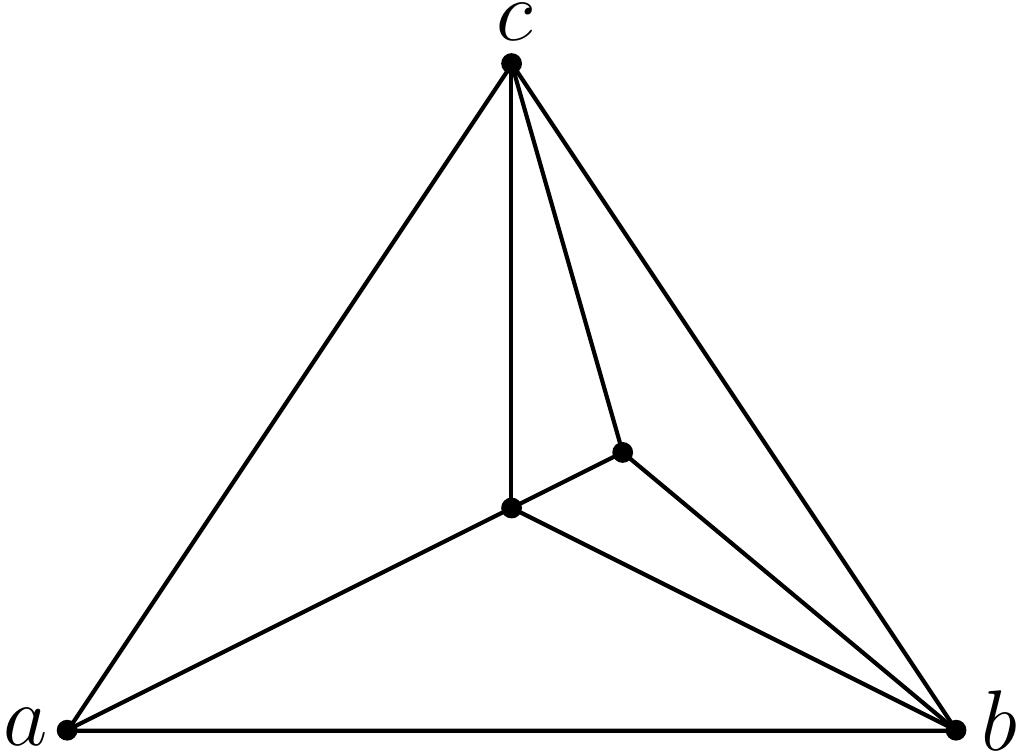}
        \caption{$\widetilde{T}_3$}
    \end{subfigure}%
    \hfill
    \caption{Three planar graphs which do not admit a simultaneous embedding with a fixed mapping for the outer face. This family is taken from \cite{CHK15}.}
    \label{fig2:family of 3}
\end{figure}

For any nonnegative integers $k_1,k_2,k_3$, we now construct a family $\g$ of graphs using graphs from the families $\mathcal{T}$ and $\widetilde{\mathcal{T}}$. Consider the skeleton graph $\skel$ shown in Figure~\ref{fig:skeleton}. The graph $\skel$ is a planar 3-tree in which the outer face is bounded by $x,y,z$, the region bounded by $y,w,z$ contains $k_1$ inner vertices, the region bounded by $z,w,x$ contains $k_2$ inner vertices, and the region bounded by $x,w,y$ contains $k_3$ inner vertices. 
\begin{figure}[H]
    \centering
    \includegraphics[scale=0.65]{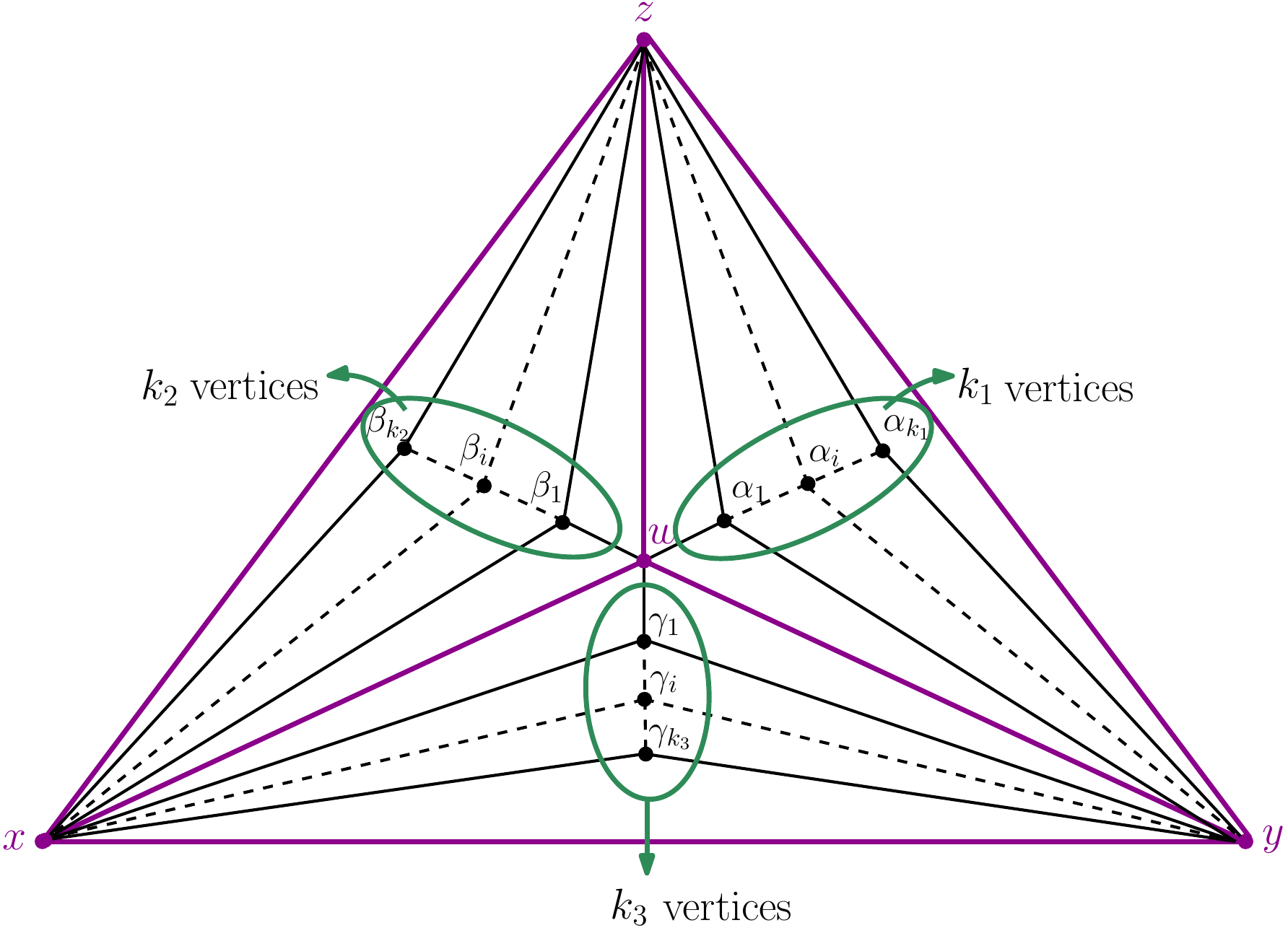}
    \caption{The skeleton graph $\skel$.}
    \label{fig:skeleton}
\end{figure}
The graph $\skel$ has $k_1+k_2+k_3+4$ vertices and $F = 2(k_1+k_2+k_3) + 4$ faces. Let $F_1$ and $F_2$ be nonnegative integers (to be determined later) satisfying $F_1 + F_2 = F - 4$. Let $\mathcal{F}$ be the collection of faces in the skeleton exactly one of whose vertices belongs to the set $\{x,y,z\}$. Note that $|\mathcal{F}| = F-4$. Consider an arbitrary but fixed partition $\mathcal{F} = \mathcal{F}_1 \cup \mathcal{F}_2$, such that $|\mathcal{F}_i| = F_i$ for $i \in \{1,2\}$.
A graph in $\g$ is obtained by planting a graph from $\mathcal{T}$ onto each face in $\mathcal{F}_1$, and a graph from $\widetilde{\mathcal{T}}$ onto each face in $\mathcal{F}_2$. A graph from $\mathcal{T}\cup \widetilde{\mathcal{T}}$ is planted onto a face $f$ of $\skel$ by identifying the set $V(f)$ of the vertices bounding $f$ with the three vertices bounding the outer face of the graph in a way that vertex $c$ is identified with the unique vertex lying in the set $\{x,y,z\}$. We do not need to specify the mapping of vertices $a$ and $b$ into $V(f)$ since $\mathcal{T}$ and $\widetilde{\mathcal{T}}$ are flip-symmetric. The total number of vertices in a graph constructed as above is given by
\begin{equation}
\label{eqn:vert}
    n = (k_1+k_2+k_3+4) + 5F_1 + 2F_2 = 11(k_1+k_2+k_3) + 4 - 3F_2.
\end{equation}
Going modulo $11$ on both sides of the above equation, we obtain
\begin{equation*}
    n \equiv 4 - 3 F_2\; (\text{mod } 11),
\end{equation*}
which further implies
\begin{equation*}
    F_2 \equiv 7n+5\; (\text{mod } 11).
\end{equation*}
Therefore, we set $F_2 = 11 \left\{ \frac{7n+5}{11}\right\}$, where for any $r\in \mathbb{R}$, $\{r\}$ denotes the fractional part of $r$.
Using this in \eqref{eqn:vert}, we obtain
\begin{equation}
\label{eqn:sum}
    k_1 + k_2 + k_3 = \frac{1}{11} (n-4) + 3 \left\{\frac{7n+5}{11}\right\},
\end{equation}
and
\begin{equation}
\label{eqn:F}
    F = 2(k_1+k_2+k_3) + 4 = \frac{2}{11}(n+18) + 6 \left\{ \frac{7n+5}{11}\right\} = \frac{2}{11} (n+18) + \frac{6}{11} F_2.
\end{equation}
Note that the assumption $n \geq 22$  guarantees that $F-4 \geq F_2$ so that $F_1$ is nonnegative.

Finally, we set $\mathcal{G}_n: = \g$, where $k_1, k_2, k_3$ are nonnegative integers that satisfy \eqref{eqn:sum} and the additional condition $k_1 \le k_2 \le k_3 \le k_1 + 1$. Such values of $k_1, k_2, k_3$ always exist and are unique; they form the most balanced ordered partition of their sum given in \eqref{eqn:sum}. Also, we denote the base skeleton for graphs in $\mathcal{G}_n$ by $B_n := B_{k_1, k_2, k_3}$.

\section{Proof of main results}
\label{sec:proof}

\begin{lemma}\label{lb}
    For $n \ge 238$, $\mathcal{G}_n$ contains at least $\frac{1}{6} \cdot 7^{(2n-8)/11}$ pairwise non-isomorphic graphs.
\end{lemma}

\begin{proof}
    From the construction, the number of graphs in $\mathcal{G}_n$ is given by
    \begin{equation*}
        |\mathcal{G}_n| = |\mathcal{T}|^{F_1} |\widetilde{\mathcal{T}}|^{F_2}.
    \end{equation*}
    It suffices to show that each isomorphism class of $\mathcal{G}_n$ contains at most six graphs since the above equation and equation \eqref{eqn:F} together imply that the number of non-isomorphic graphs in $\mathcal{G}_n$ is bounded below by
    \begin{equation*}
        \frac{1}{6} \cdot |\mathcal{G}_n| = \frac{1}{6} \cdot 7^{(2n-8-5F_2)/11}\  3^{F_2} = \frac{1}{6}  \cdot 7^{\frac{2n-8}{11}} \left(\frac{3}{7^{5/11}}\right)^{F_2} \ge \frac{1}{6} \cdot 7^{\frac{2n-8}{11}}.
    \end{equation*} 
    We now bound the size of any isomorphism class of $\mathcal{G}_n$. Let $G \in \mathcal{G}_n$. Suppose that $H \in \mathcal{G}_n$ is isomorphic to $G$. Let $\varphi$ be an isomorphism from $G$ to $H$. Note from the construction in the previous section that the degree of the vertex $z$ satisfies the inequality
    \begin{align*}
        \text{deg}(z) \ge 2(k_1 + k_2) + 3 \ge 2 \cdot \frac{2}{3} \left(\frac{n-4}{11} - 1\right) + 3 > 30. 
    \end{align*}
    The same inequality also holds for vertices $x$ and $y$. The degree of any other vertex $v$ belonging to the underlying skeleton graph $B_n$ is bounded above by
    \begin{equation*}
        \text{deg}_{B_n}(v) + 4 \times \text{number of faces incident to $v$ in } B_n,
    \end{equation*}
    which is at most $6 + 4 \cdot 6 = 30$. Finally, the degree of an internal vertex of a face of $B_n$ is bounded above by $7$. Since graph isomorphisms preserve degrees, we conclude that $\varphi$ maps $\{x, y, z\}$ to $\{x, y, z\}$. Therefore, there are $3!$ possibilities for the map $\varphi$ restricted to the set $\{x, y, z\}$. We will show that fixing $\varphi|_{\{x, y, z\}}$ yields at most one possibility for the graph $H$.

    Suppose that $\sigma \in S_3$ is the permutation that corresponds to the mapping $\varphi|_{\{x, y, z\}}$, where $x$, $y$, and $z$ are indexed as the first, second, and third symbols, respectively. Note that $w$ is the unique vertex incident to all of $x$, $y$, and $z$. Therefore, $\varphi(w) = w$. Using a similar idea, it follows from an inductive argument that $\varphi(\xi_i) = \varphi(\sigma(\xi)_i)$, where $\xi \in \{\alpha, \beta, \gamma\}$ and $\alpha$, $\beta$, and $\gamma$ are indexed as the first, second, and third symbols, respectively. Therefore, $\varphi$ maps $B_n$ to $B_n$. Note that no two graphs in $\mathcal{T}$ (similarly, $\widetilde{\mathcal{T}}$) are isomorphic, provided that the isomorphism fixes the vertices of the outer face. Now it is clear that for any face $a, b, c$ of $B_n$, the graph planted onto $a, b, c$ in $G$ must be the same as the graph planted onto the face $\varphi(a), \varphi(b), \varphi(c)$ in $H$. Hence, there is at most one possibility for $H$.   
\end{proof}

\begin{remark}
    When $k_1, k_2, k_3$ are not all equal (exactly two of them are equal), it is possible to improve the constant in the lower bound in Lemma~\ref{lb}. In this case, the proof of the lemma implies that there are at most $2!$ graphs in each isomorphism class of $\mathcal{G}_n$.
    % exists a graph $H$ if and only if the set of non-fixed points of $\sigma$ is a subset of the set
    % \begin{equation*}
    %     \{i \in [3]: k_i = \text{mode}(\{k_1, k_2, k_3\})\} \textcolor{red}{???}.
    % \end{equation*}
\end{remark}

The following lemma will be crucial in the proof of the upper bound (Lemma~\ref{ub}).

\begin{lemma}
\label{lem:partial}
    Let $G$ be a planar $3$-tree with $n$ vertices, and $X \subset \R^2$ be a set of $n$ points in general position with convex hull $\mathrm{Conv}(\{p_1, p_2, p_3\})$. Suppose that the face bounded by $a, b, c$ is the outer face of $G$, with the partial straight-line embedding $\phi$ of $G$ on $X$ given by $\{a \mapsto p_1, b \mapsto p_2, c \mapsto p_3\}$. Further, suppose that $d$ is a vertex adjacent to $a$, $b$, and $c$ such that the triangle bounded by $b, c, d$ contains $n_1$ vertices, the triangle bounded by $a, c, d$ contains $n_2$ vertices, and the triangle bounded by $a, b, d$ contains $n_3$ vertices. Then there is at most one choice for the point $\phi(d)$ while completing $\phi$ to a straight-line embedding.
\end{lemma}

\begin{proof}
    % Note that the number of points in $X$ that are contained in the triangle $p_2, p_3, \phi(d)$ is $n_1$, the number of points inside the triangle $p_1, p_3, \phi(d)$ is $n_2$, and the number of points inside the triangle $p_1, p_2, \phi(d)$ is $n_3$. 
    It suffices to show that there is at most one point $p \in X$, such that the number of points of $X$ lying inside the triangle $p_2, p_3, p$ is $n_1$, the number of points inside the triangle $p_1, p_3, p$ is $n_2$, and the number of points inside the triangle $p_1, p_2, p$ is $n_3$. Suppose there are two such points, say $p'$ and $p''$. Without loss of generality, we may assume that $p''$ lies inside the triangle $p_1, p_2, p'$. Then the number of points inside the triangle $p_1, p_2, p''$ is strictly less the number of points inside the triangle $p_1, p_2, p'$, a contradiction.
\end{proof}

\begin{lemma}\label{ub}
    For $n \ge 238$, at most $(21n + 552) 4^{(n+37)/11}$ pairwise non-isomorphic graphs in $\mathcal{G}_n$ are simultaneously embeddable.
\end{lemma}

\begin{proof}
    Let $\mathcal{G} \subseteq \mathcal{G}_n$ consist of pairwise non-isomorphic graphs that are simultaneously embeddable on an $n$-point set $X$ in the plane. Since $\mathcal{G}$ is a class of maximal planar graphs, the convex hull $\text{Conv}(X)$ of $X$ must be a triangle. Without loss of generality, we may assume that no three points in $X$ are collinear. Otherwise, we can slightly perturb the point set $X$ so that the same embeddings remain valid. 
    
    Let $G \in \mathcal{G}$. Let $f$ denote the face of $G$ that maps to the vertices of $\text{Conv}(X)$. We partition $\mathcal{G}$ into classes based on the face $h$ of $B_n$ containing $f$, the graph $T$ planted in the face $h$ (with the convention that $T$ is a $3$-cycle when $h$ is one of $\langle x, y, z \rangle$, $\langle x, y, \gamma_{k_3} \rangle$, $\langle \alpha_{k_1}, y, z \rangle$, and $\langle x, \beta_{k_2}, z \rangle $), the face of $T$ that corresponds to $f$, and the mapping of the vertices of $f$ to the vertices of $\text{Conv}(X)$. This divides $\mathcal{G}$ into at most $6(11 \cdot 7 \cdot F_1 + 5 \cdot 3 \cdot F_2 + 1 \cdot 1 \cdot 4)$ partition classes. The three terms in the above expression correspond to the cases when $T \in \mathcal{T}$, $T \in \widetilde{\mathcal{T}}$, and $T$ is a $3$-cycle, respectively.

    Let $\mathcal{C}$ be one of these partition classes. We claim that for all graphs in $\mathcal{C}$, their straight-line embeddings on $X$ map the vertices of $B_n$ as well as the internal vertices of the graph planted in face $h$ to the same points of $X$. Let $G \in \mathcal{C}$ and let $H$ be the induced subgraph of $G$ on the vertices of $B_n$ along with the internal vertices of the graph planted in face $h$. The graph $H$ is the same for all graphs in $\mathcal{C}$. Furthermore, the number of internal vertices of $G$ in each inner face (a face other than $f$) of $H$ is also the same for all graphs in $\mathcal{C}$. Noting that $H$ is a $3$-tree, the claim follows from a straightforward inductive argument using Lemma~\ref{lem:partial}. We remark that $H$ can be viewed as a new skeleton for graphs in the partition class $\mathcal{C}$. Figure~\ref{fig:H} shows an example of a graph $G$ on $22$ vertices and the induced subgraph $H$ for all graphs in the corresponding partition class.

    \begin{figure}[H]
        \centering
        \hspace{25pt}
        \begin{subfigure}{0.4\textwidth}
            \centering
            \includegraphics[width=\textwidth]{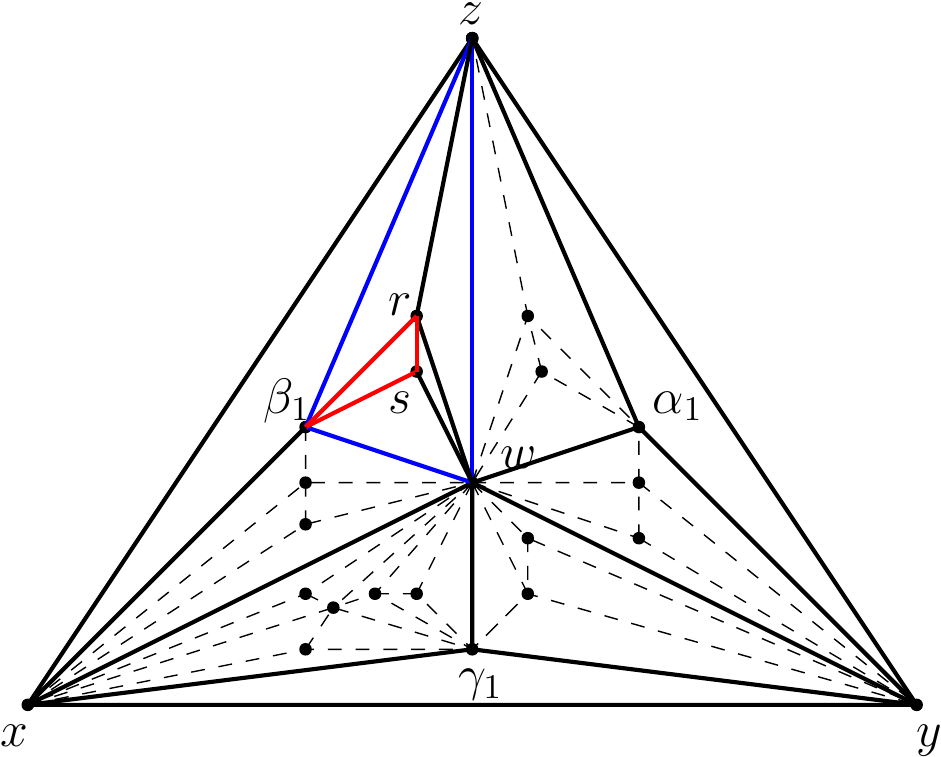}
            % \caption{$T_6$}
        \end{subfigure}
        \hfill
        \begin{subfigure}{0.4\textwidth}
            \centering
            \includegraphics[width=\textwidth]{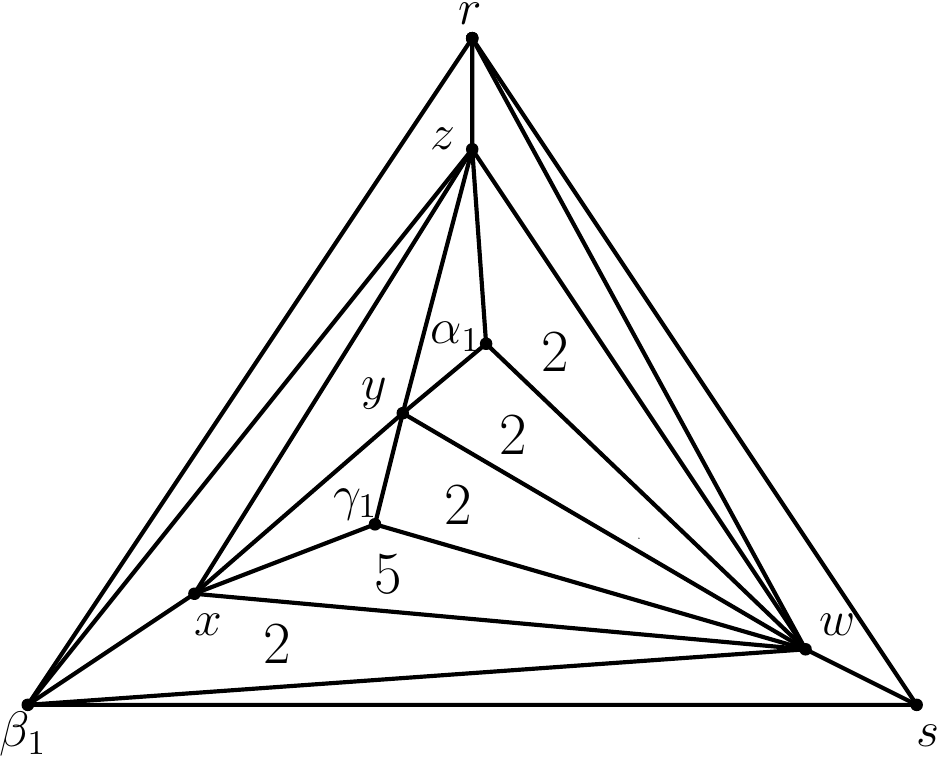}
            % \caption{$T_7$}
        \end{subfigure}
        \hspace{25pt}
        \caption{On the left is a graph $G \in \mathcal{G}_{22}$ with $\beta_1, s, r$ as the face $f$ (red edges) and $\beta_1, w, z$ as the face $h$ (blue edges). The edges of the induced subgraph $H$ are solid, while the remaining edges are dashed. On the right is the induced subgraph $H$ common to all graphs in the corresponding partition class. Faces are marked with the number of internal vertices if the latter is nonzero.}
        \label{fig:H}
    \end{figure}

    It follows that the graph planted in a fixed face $g \ne h$ of $B_n$ must be embedded on the same subset of $X$ for all graphs in $\mathcal{C}$ since the mapping for the outer face of $g$ is already fixed. But we know that at most two graphs in $\mathcal{T}$ (similarly, $\widetilde{\mathcal{T}}$) can be simultaneously embedded with a fixed mapping for the outer face. Therefore, we conclude that
    \begin{equation*}
        |\mathcal{C}| \le 2^{F_1 + F_2},
    \end{equation*}
    which further implies
    \begin{align*}
        |\mathcal{G}| \le 6(77 \cdot F_1 + 15 \cdot F_2 + 4) 2^{F_1 + F_2} &\le 6(77(F-4) + 4) 2^{F-4},\\
        &\le 6 \left(77 \cdot \frac{2n+52}{11} + 4\right) 2^{(2n+52)/11},\\
        &= (21n + 552) 4^{(n+37)/11},
    \end{align*}
    where we used \eqref{eqn:F} and the fact $F_2 \leq 10$.
    % We index the inner faces of $B_n$ by $i$. And for each inner face of $B_n$, we have a $3$-tree from $\mathcal{T} \cup \widetilde{\mathcal{T}}$. We index the inner faces of each $3$-tree in $\mathcal{T} \cup \widetilde{\mathcal{T}}$ by index $j$. The index $(i,j)$ uniquely identifies the inner faces of graphs in $\mathcal{G}$. We index the outer face of $G$ by $(0,0)$. The total number of faces in $G$ is equal to $11 F_1 + 5 F_2 + 1$. We partition $\mathcal{G}$ according to the index of the face that maps to the vertices of $\text{Conv}(X)$ and the precise mapping of this face. There are $3!$ ways to map vertices of this face to those of $\text{Conv}(X)$. Therefore, we have a total of 
\end{proof}

We end the paper by presenting the proof of Theorem~\ref{thm} and Corollary~\ref{cor}.

\begin{proof}[Proof of Theorem~\ref{thm}]
It follows from Lemma~\ref{lb} and Lemma~\ref{ub}. The statement of the theorem holds with a maximal collection of non-isomorphic graphs in $\mathcal{G}_n$ renamed to $\mathcal{G}_n$. 
\end{proof}

\begin{proof}[Proof of Corollary~\ref{cor}]
Let $n$ be sufficiently large. Consider the collection $\mathcal{G}_n$ from Theorem~\ref{thm}. We use a double counting argument to give a lower bound on $f(n)$.

Let $X$ be a point set with size $m$ so that all graphs in $\mathcal{G}_n$ admit a straight-line embedding on $X$. In particular, for each graph $G \in \mathcal{G}_n$, there is a subset $Y \subset X$  with size $n$, such that $G$ admits a straight-line embedding on $Y$. On the other hand, by Theorem~\ref{thm}, for any subset $Y \subset X$ with size $n$, at most $(21n + 552) 4^{(n+37)/11}$ graphs from $\mathcal{G}_n$ admit a straight-line embedding on $X$. Therefore, we conclude that
\begin{equation*}
    \binom{m}{n} (21n + 552) 4^{(n+37)/11} \geq |\mathcal{G}_n|\geq \frac{7^{(2n-8)/11}}{6}.
\end{equation*}
Taking logarithms and dividing by $n$ on both sides of the above inequality, we get
\begin{equation}
\label{eqn:double}
    \frac{1}{n}\log \binom{m}{n}+ \frac{2}{11} \log 2+o(1) \geq \frac{2}{11}\log 7+o(1).
\end{equation}
By Stirling's approximation, we have
\begin{align*}
    \log \binom{m}{n} &= m\log m-n \log n-(m-n)\log (m-n)+O(\log m),\\
    &= n \log \frac{m}{n}+(m-n)\log \frac{m}{m-n}+O(\log m).
\end{align*}
Assuming $\log m = o(n)$ and using the above estimate in \eqref{eqn:double} yields
\begin{equation*}
    \log \frac{m}{n}+\bigg(\frac{m}{n}-1\bigg)\log \frac{m}{m-n}+o(1) \geq \frac{2}{11}(\log 7-\log 2),
\end{equation*}
and thus $f(n) \geq m \geq (1.059-o(1))n$.
\end{proof}

\section*{Acknowledgement}
The authors thank J\'ozsef Solymosi for helpful discussions.

\end{document}